\documentclass[12pt]{amsart}
\usepackage{amscd}
\usepackage[all]{xy}
\usepackage[notcite,notref]{showkeys}

 \newcommand{\Br}{\operatorname{Br}}

\newcommand{\rank}{{\operatorname{rank}}}

\newcommand{\rat}{{\operatorname{rat}}}

\newcommand{\CH}{\operatorname{CH}}

\newcommand{\Pic}{\operatorname{Pic}}

\newcommand{\Bl}{\operatorname{Bl}}

\newcommand{\Grass}{\operatorname{Grass}}
\newcommand{\supp}{\operatorname{supp}}

\newcommand{\C}{\mathbf{C}}
\renewcommand{\L}{\mathbf{L}}
\renewcommand{\P}{\mathbf{P}}
\newcommand{\N}{\mathbf{N}}
\newcommand{\Q}{\mathbf{Q}}
\newcommand{\Z}{\mathbf{Z}}
\newcommand{\un}{\mathbf{1}}

 \newcommand{\sL}{\mathcal{L}}
 \newcommand{\sC}{\mathcal{C}}

\newcommand{\sM}{\mathcal{M}}

\newcommand{\sO}{\mathcal{O}}

 \newcommand{\sF}{\mathcal{F}}
\newcommand{\sA}{\mathcal{A}}
 \newcommand{\sI}{\mathcal{I}}
 \newcommand{\sN}{\mathcal{N}}
 \newcommand{\sS}{\mathcal{S}}

 \newcommand{\sT}{\mathcal{T}}

 \numberwithin{equation}{section}

\theoremstyle{plain}
\newtheorem{thm}[equation]{Theorem}
\newtheorem{prop}[equation]{Proposition}
\newtheorem{lm}[equation]{Lemma}

\newtheorem{conj}[equation]{Conjecture}

\theoremstyle{definition}
\newtheorem{defn}[equation]{Definition}
\newtheorem{ex}[equation]{Example}
\newtheorem{rk}[equation]{Remark}

\begin{document}
 \title{K3 surfaces associated to a cubic fourfold }
 \author[C. Pedrini]{Claudio Pedrini}
\address{Dipartimento di Matematica \\ %
Universit\'a degli Studi di Genova \\ %
Via Dodecaneso 35 \\ %
16146 Genova \\ %
Italy}
\email{pedrini@dima.unige.it}

\begin{abstract}Let $X\subset \P^5$ be a smooth cubic fourfold. A well known conjecture asserts that $X$ is rational if and only if there an Hodge theoretically associated K3 surface $S$. The surface $S$ can be  associated to $X$ in two other different ways. If there is an equivalence of categories $\sA_X \simeq D^b(S,\alpha)$ where $\sA_X$ is the Kuznetsov component of $D^b(X)$ and $\alpha$ is a Brauer class,  or if there is an isomorphism between the transcendental motive $t(X)$ and the (twisted ) transcendental motive of  a K3 surface$S$. In this note we consider families  of cubic fourfolds  with a finite group of automorphisms and describe the cases where there is an associated  K3 surface in one of the above senses.\end{abstract}
\maketitle 

\section{Introduction}

Let $X$ be a complex smooth cubic fourfold in $\P^5$. Let $A(X)=H^4(X,\Z)\cap H^{2,2}(X)$ be the lattice of algebraic cycles. The fourfold $X$ is {\it special} if  the lattice $A(X)$ contains a class  $v$ that is not homologous to $h^2$, where $h$ is the class of a hyperplane section. A fourfold $X$ is  {\it very general} if $\rank A(X) =1$. For $X$ special let $ K_d=<h^2,v>$, where $d$ is the discriminant of $<h^2,v>$.\par
 A polarized K3 surface  $(S,L)$ of  degree $d$ and genus $g$,  with $L^2= 2g-2$ and $2g-2 =d$, is  said  to be associated to a cubic fourfold $X \in \sC_d$ if there is an isomorphism of Hodge structures
\begin{equation} \label{K3} K^{\perp}_d \simeq H^2(S, \Z)_{prim}(-1). \end{equation}
Let $\sC$ be the moduli space of smooth projective cubic fourfolds and let  $\sC_d \subset \sC$ be the {\it Hassett divisor }  parametrizing  special  cubic fourfolds  with a labelling of discriminant d, i.e. a positive defined rank-two primitive sublattice $K_d=<h^2,v> \subset A(X)$.
 B. Hassett  in [Hass] proved  that $X \in \sC_d$ has an associated  polarized K3 surface  of degree $d$ if and only if $d$ satisfies the following numerical condition\par
(*) $d>6$ and  $d$ is not divisible by 4,9 or a prime $p \equiv 2 (3)$.\par
More precisely: for $d$ satisfying (*)  $\sC_d$ is birational to $\sN_d$ if $d \equiv 2 (6) $, where $\sN_d$ is the moduli space of polarized K3 surfaces of degree d. Otherwise $\sC_d$ is birational to a quotient of $\sN_d$ by an involution, see [Hass, Cor.25].\par
The results of several Authors (Hassett,Kuznetsov, Addington-Thomas) suggest
that a cubic fourfold is rational if and only if it has an associated
K3 surface. Kuznetsov conjectured that X has an associated K3 surface
if and only if there exists a semi-orthogonal decomposition of the
derived category $D^b(X)$ of bounded complexes of coherent sheaves
\begin{equation}\label{Kuz} D^b(X)=< \sA_X, \sO_X(1),\sO_X(2)>\end{equation}
such that $\sA_X,$ is equivalent to the category $D^b(S)$,where $ S$ is K3 surface.The following result proved in [BLMNPS, Cor.29.7], shows that the two conditions are in fact equivalent
\begin{thm}\label{FMequiv} Let $X$ be cubic fourfold. Then $X$ has a Hodge-theoretically associated K3 surface if and only if there exists a smooth projective K3  surface $S$ and an equivalence $\sA_X \simeq D^b(S)$\end{thm}
These results leaded to the following Conjecture
\begin{conj} \label{rat}  A smooth cubic fourfold $X$ is rational if and only if $X \in \sC_d$  with $d$ satisfying the  numerical condition in (*) \par
\end{conj}
If  $X$ is a very general cubic fourfold , i.e. $CH^2(X) \simeq \Z$, then $X$ has no associated K3 surface and hence  is conjecturally irrational.\par
In the particular case where $d$ is an integer satisfying (*) which is of the form 
$$d = 2(n^2+n+1) \   with  \    n \in \N$$
\noindent there is a  rational map of degree 2
\begin{equation} \label{moduli} \phi :  \sF_g  \to \sC_d, \end{equation}
\noindent  where  $\sF_g$ is  the 19-dimensional moduli space of smooth polarized K3 surfaces $(S,H)$ of genus $g$ with  $H^2 =2g -2$ and  $g = n^2 +n +2$. The map $\phi$ sends $(S,H)$ to the class  $\vert X\vert $ of a cubic fourfold $X \in \sC_d$. The surface $S$ is associated to the cubic fourfold $X$  and there is an isomorphism
\begin{equation} \label {Fano} S^{[2]} \simeq F(X),\end{equation}
\noindent where $S^{[2]}$  is the  Hilbert scheme of two points on  $S$.\par
Spelling out the numerical condition in (*) one finds that the above conjecture predicts rationality of all cubics in the Hasset divisors $\sC_d$, for $d =14,26,38,42,62,74,78,86,98,....$\par
In [RS] it has been proved that  a general  $X \in \sC_d$ with $d =14, 26,38,42$ is rational. In all these cases the  construction of a birational map $\mu : X \dashrightarrow W$, where $W$ is either $\P^4$ or a notable smooth rational  fourfold, can be  achieved by considering a surface $S_d \subset X$ admitting a four dimensional family of $(3e-1)$-secant curves of degree $e\ge 2$. These curves are parametrized by a rational variety with the property that trough a general point of $\P^5$ there passes a unique curve of the family. The inverse $\mu^{-1}: W \dashrightarrow X$ is given by a linear system of divisors in $\vert \sO_W(e\cdot i(W) -1)\vert$, where $i(W)$ is the index of $W$, having   points of multiplicity $e$ along  an irreducible surface $U \subset W$, which is a birational incarnation of a K3 surface $S$ associated to $X$.\par
The results by Hassett and Kuznetsov have been extended by  B\"ulles in [Bull] showing that if  $X\in \sC_d$, where $d$ satisfies the following numerical condition 
$$(**)  \exists f,g \in \Z   \   with \   g\vert 2n^2+2n+2 \  ,  \ n \in \N  \  and   \ d =f^2g $$
\noindent then there is a equivalence of categories
\begin{equation} \label{Brauer} \sA_X \simeq D^b(S,\alpha),\end{equation}
 \noindent with $S$ a K3 surface and $\alpha$ a Brauer class in $\Br(S)$. Clearly 
 $$\bigcup_{d\in (*) }\ne \bigcup_{d \in (**)}, $$
For instance,  if the fourfold $X$ contains a plane, and hence  it belongs to $\sC_8$ , then  there is a K3 surface associated  to $X$ in the sense of \ref{Brauer}, but , for a general $X \in \sC_8$, not   in the Hodge theoretically sense.\par
Let $X$ be a special cubic fourfold and let $A(X)$ be the lattice of algebraic cohomology,  i.e.
$$A(X) = H^4(X,\Z) \cap H^{2,2}(X)$$
If  $\rank A(X) \ge 12 $  then $X$ has an associated K3 surface  in the Hodge theoretically sense, see Remark. 2.11.\par
 A third way in which a K3 surface  can be possibly  associated  to a cubic  fourfold is via the transcendental motives.  Let $\sM_{rat}(\C)$ be the (covariant) category of Chow motives. The motive $h(X)$ of a cubic fourfold $X$ has a {\it reduced Chow K\"unneth decomposition }
\begin{equation}\label {decomp} h(X) = \un \oplus \L \oplus (\L^2)^{\oplus \rho(X)} \oplus t(X) \oplus \L^3 \oplus \L^4,\end{equation}
\noindent where  $\L$ is the Lefschetz motive, $\rho(X) =\rank A^2(X)$, with $A^2(X)= CH^2(X)\otimes \Q$ and $t(X)$ is  the transcendental motive, see [BP]. Then
$$H^*(t(X)) = H^4_{tr}(X,\Q) =T(X)_{\Q},$$
where $T(X)\subset H^4(X.\Z)$   is the transcendental lattice and
$$A^i(t(X)) =0   \  for   \  i\ne 3  \  ;   \  A^3(t(X)) =A_1(X)_{hom} =A_1(X)_{alg},$$
Similarly if $S$ is a K3 surface then its motive $h(S)$ has a reduced Chow-K\"unneth  decomposition as follows
$$h(S)=\un \oplus \L^{\oplus \rho(S)}\oplus t_2(S) \oplus \L^2$$
\noindent where $\rho(S)$ is the rank of the Neron-Severi $N(SX)$, see [KMP].
\begin{defn} \label{motivic} The motive $h(X)$ of a special cubic fourfold $X$  is {\it associated to the motive of a K3 surface $S$} if there is an isomorphism of motives in the (covariant) category $\sM_{rat}(\C)$
\begin{equation}\label{K3iso} t(X) \simeq t_2(S)(1),\end{equation} 
\noindent where $t_2(S)(1) = t_2(S) \otimes \L$.\end{defn}
The  equivalence of categories in \ref{Brauer} implies that there is an isomorphism of transcendental motives, as in \ref{K3iso}, see [Bull]. For a very general cubic fourfold no such  isomorphism can exists, see [BP, Prop.3.4]. \par
The existence of an isomorphism as in  \ref{K3iso} implies that  the motive $h(X)$ is finite-dimensional, in the sense of [Kim], if and only if $ h(S) $ is such.\par
Some families of cubic fourfolds with a finite-dimensional motive have been described in [Lat 1] and [Lat 3]. In all these cases the motive  $h(X)$ is of abelian type, i.e. lies in subcategory  $\sM^{Ab}_{rat}(\C)$ generated by the motives of abelian varieties.\par
If a cubic fourfold $X$ has  a finite  group $G$ of symplectic automorphisms whose order order  is $\ne 1, 2$,  then there is a unique K3 surface  $S$ associated to $X$ in the sense of \ref{Kuz}, i.e. $\sA_X \simeq D^b(S)$,
see [Ou, Thm.8.4]. Therefore $X$ is conjecturally rational and there is an isomorphism of motives as in \ref{K3iso}.\par
In this note we consider families of cubic fourfolds admitting either an involution or a non-symplectic automorphism of prime order. Note that a non-symplectic  prime order automorphism can have order 2 or 3. \par
In Section 2  we show that  for every  cubic fourfold  $X$ with an involution  the motive $h(X) $  is associated to the motive of  a K3 surface $S$ in the sense of \ref{K3iso}. For  two distinct families of cubic fourfolds  $X$ with an involution, there is no K3 surface  Hodge theoretically associated to $X$, while there is one family consisting of cubic fourfolds  with an involution that have an associated K3 surface, are rational and their motives are of abelian type.\par
In Sect 3 we deal with cubic fourfolds  $X$ admitting a non -symplectic automorphism of order 3 and prove that there are 2  cases where $X$ has an associated K3 surface and is rational  and 2 cases where $X$ is conjecturally irrational.\par
In  Sect. 4 we describe families of special cubic fourfolds whose associated K3 surfaces are singular, i.e. their Picard rank is 20. These special cubic fourfolds have amotive of abelian type.\par
In Sect.5   we restrict to the case of cubic fourfolds $X$ such that $F(X) \simeq S^{2]}$, with $S$ a K3 surface, and show that  the motive $h(X) $  is associated to the motive of $S$ in the sense of \ref{K3iso}. We also show that a symplectic automorphism  $ \sigma$ on $S$
induces a symplectic automorphism $\tau $ on $X$ and give examples, in the case $\sigma$ has order 3, of K3 surfaces  $S$ whose motive is associated to different cubic fourfolds $X_1 \in \sC_{d_1} $ and  $X_2 \in \sC_{d_2}$ with $d_1 \ne d_2$.\par

\section{Cubic fourfold with an involution}
The group of regular automorphisms of a cubic fourfold  $X\subset \P^5$ is  finite and any automorphism $f $  of $X$ is induced from a linear automorphism of the ambient projective space. Therefore the induced automorphism  $f^*$ on $ H^4(X,\Z)$ preserves 
 $h^2 $, where $h$ is the class of a hyperplane section. The Abel-Jacobi map 
 $$p_*q^* : H^4(X,\Z) \simeq H^2(F(X),\Z),$$
 \noindent where $P_F \subset F(X) \times X$ is the universal family and $p,q$ are the projections to $F(X)$ and $X$, respectively, induces an isomorphism of Hodge structures 
 $$H^{3,1}(X) \simeq H^{2,0}(F(X)).$$
The automorphism $f$ is symplectic if $f^*$ acts as the identity on $H^{2,0}(F(X)$, that is equivalent to $f^*$ acting  as the identity  on $H^{3,1}(X) \simeq \C\omega$.  If $f$ is non-symplectic and has order $n$ then 
$$f^*(\omega)= \zeta \omega,$$
\noindent with $\zeta$ a primitive $n^{th}$-root of 1.  Since $f$ has finite order we can assume that it acts diagonally, that is 
$$f(x_0,x_1,x_2,x_3,x_4,x_5)= (a_0x_0, a_1 x_1,a_2 x_2, a_3 x_3,a_4 x_4,a_5x_5).$$
Let's  denote by  
$$\det( f )= \prod _i a_i$$
Let $F(x_0,x_1,x_2,x_3,x_4,x_5)=0$ be the equation of $X \subset \P^5$. If $X$ is invariant under the induced action of the automorphism $f$  there is a constant $\lambda_f \in \C^*$ such that
\begin{equation} \label{action} f^*(F) = \lambda_f F \end{equation} 
The action of $f$ is symplectic on $F(X)$ if and only if 
$$\det (f) / \lambda^2_f =1,$$
\noindent see [BCS 1, Lemma 6.2].\par
There are three types of involutions on a cubic fourfold $X\subset \P^5$. Two of them are non-symplectic and
one is symplectic, see [Marq]. An involution $\phi$ is symplectic  if the induced automorphism on $H^4(X,\Z)$ acts trivially on $H^{3,1}(X)$,  it is non-symplectic
if it acts as -1. Applying a linear change of coordinates we can assume that the involutions are the following ones
$$\phi_1 : [x_0,x_1,x_2,x_3,x_4,x_5] \to [x_0,x_1,x_2,x_3,x_4,-x_5] ;$$
$$ \phi_2 :  [x_0,x_1,x_2,x_3,x_4,x_5] \to [x_0,x_1,x_2,x_3, - x_4,-x_5] ; $$
$$\phi_3 :  [x_0,x_1,x_2,x_3,x_4,x_5] \to [x_0,x_1,x_2,- x_3,-x_4,-x_5] .$$
\noindent The involution $\phi_2$ is symplectic, while $\phi_1$ and $\phi_3$ are non-symplectic.\par
The family $\sF_1$ of cubic fourfolds which are invariant under the involution $\phi_1$ has dimension 14 and coincides with the family of cubic fourfolds with an Eckardt point. If a  smooth cubic 4-fold $X \subset \P^5$ contains an Eckardt point $p$ then we can choose
coordinates $[x_0,\cdots, x_5] $ in $\P^5$ such that  $p=(0,0,0,0,0,1)$ and $X$ is defined by an equation
\begin{equation}  \label{Eckardt}    f(x_0,x_1,x_2,x_3,x_4) +  l(x_0,x_1,x_2,x_3,x_4) x^2_5=0 \end{equation}
\noindent with $f$ of degree 3 and $l$ of degree 1,see [LZP,Lemmas 1.6 and 1.8].
\begin{prop} Let $X$ be a  cubic fourfold with an Eckardt point..There is a K3 surface $S$ with Picard rank 6 such that $t_2(S)(1) \simeq t(X)$. The K3 surface $S$ is not associated to $X$ in the Hodge theoretical sense.\end{prop}
\begin{proof} The equation  $f (x_0,x_1,x_2,x_3,x_4)=0$  in \ref{Eckardt} defines a smooth cubic 3-fold $Y \subset \P^4$, with $\P^4 : (x_5=0)$. Let $H \subset \P^4$ be the hyperplane defined 
by $l(x_0,x_1,x_2,x_3,x_4)=0$. The intersection $S =Y \cap H$ is a cubic surface. The 27 planes $\P^2_{<l_i,p >}$, where $(l_1\cdots,l_{27})$ are the lines on $S$, are contained in $X$. The lattice of the  primitive algebraic cohomology $A(X)_{prim}$ is isomorphic to $E_6(2)$ and is spanned by the classes $[P_i] -[P_j]$,with $P_i,P_j$ planes trough $p$, see [Marq,Thm.3.4]. There is a class $[F_0]\in A(X)$ such that $[F_0] \cdot [F_0]=7$ and
$[F_0] \cdot h^2 = 3$, see [LZP, Lemma 2.4]. Therefore the rank two lattice $<h^2,]F_0]>$ has discriminant 12. Since $X$  contains a plane we get $X \in \sC_8 \cap \sC_{12}$.\par 
 Let $m$ be a line in the cubic $S =Y \cap H$ and let $P=\P^2_{<m,p>} \subset X$ be the plane generated by $m$ and the Eckardt poin $p$. Projecting $X$ from $P$ to a plane $P'$ which is complementary to $P$ in $\P^4 : (x_5=0)$ gives a quadric bundle
\begin{equation} \pi: \tilde X \to P'=\P^2\end{equation}
\noindent whose discriminant locus is a degree 6 curve $D$. The curve $D$ consists of two irreducible components:  a quintic curve $C$ and the line $L =H \cap P'$.
The quintic curve $C$ is the discriminant divisor of the conic bundle $\tilde Y \to P'$, obtained by blowing-up the line $m$ on the cubic 3-fold $Y$. The fibre of $\pi$ over a singular point in $C \cap L$ is a couple of planes. Let $T \to \P^2$ be the double cover parametrizing the lines in the fibres of $\pi$. Let $F(\tilde X/\P^2)$ be the relative Fano variety of lines  in  the fibration $\pi$ and let $\sL \to F(\tilde X/\P^2)$ be the universal line  sitting in the diagram
\begin{equation}\CD \sL @>{q}>> \tilde X \\
@V{p}VV    @V{\pi}VV  \\
F(\tilde X/\P^2)@>>>\P^2 \endCD \end{equation}
 In the Stein factorization of the composite map $\sL \to \P^2$
 \begin{equation} \sL \to T \to \P^2\end{equation} 
\noindent the map $\sL \to T$ is a $\P^1$ bundle and $T \to \P^2$ is a double cover ramified along the reduced sextic $C \cup L$. Let $\tau : \tilde \P^2 \to \P^2$ be the blow-up of $\P^2$ at the 5 nodes in $C \cap L$ and let $S \to \tilde \P^2$ be the double cover ramified along the strict  transform of $D =C\cup L$. Then $S$ is a K3 surface and $\tau\vert_S : S \to T$ is a desingularization. Let 
$$f : S \to T \to  \P^2$$
\noindent and let  $h =f^*[l]$ be the pull-back of the class of a line from $\P^2$. Since the surface $T$ has five ordinary double points coming from the intersection  $C \cap L$ it follows that $\Pic S$ contains five additional classes corresponding to the exceptional divisors of $S \to T$. The line $L$ is in the branch locus of $f$ and hence the surface $S$ contains another rational curve, namely the inverse image $L'$ of the line $L$. 
Therefore the K3 surface $S$ contains 6 irreducible (-2)-curves whose classes $l',e_1,\cdots, e_5$ span a rank 6 lattice $M$ which admits a unique primitive embedding 
in $\Pic S$. The cubic fourfold $X$ is not associated to the K3 surface $S$ in the  Hodge theoretically sense see   [Laz , Rk.3.9].\par 
Since $X$ contains a plane there is an isomorphism of motives
$$t_2(S) (1) \simeq t(X),$$
\noindent see [Bull, Ex 3.3].
\end{proof}

\begin{prop} Let $\sF_2$ be the family of cubic fourfolds that are invariant under the involution $\phi_2$. If $X \in \sF_2$ there is a K3 surface $S$,  double cover of a cubic surface branched along a degree 6 curve, such that
$$t(X) \simeq t_2(S)(1)$$
The cubic $X$  belongs to $\sC_{12}$, contains no plane  and has no Hodge theoretically  associated K3 surface,
\end{prop}
\begin{proof} A cubic $X$ invariant under $\phi_2$ has an equation of the form
\begin{equation} f(x_0,x_1,x_2,x_3)+x^2_4 l_1+x^2_5 l_2 +x_4x_5l_3=0,\end{equation}
\noindent where $f$ is homogeneous of degree 3  and $l_i$ are homogeneous of degree 1 in $x_0,x_1,x_2,x_3$. The involution $\phi_2$  is symplectic. The locus of fixed points of $\phi_2$ on $\P^5$ is the disjoint union of a $\P^3$ defined by $x_4=x_5=0$
and the line $r$ joining the base points $P_4$ and $P_5$. The line $r$ is contained in $X$ and the fixed locus on $P^3$ is the cubic surface  $W :  f(x_0,x_1,x_2,x_3)=0$.The fixed locus of $\phi_2 $ on $F(X)$ consists of  the line $r$, the 27  lines on a cubic surface $W$ and a K3 surface $S$, see [Cam]. The surface $S$ parametrizes  the lines joining a fixed point $Q_1$ on $\P^3$ and a point $Q_2 $ on $r$. The surface $S$ is a double cover of the cubic surface $W$ branched along a degree 6 curve , see [BP, rk.4.4].\par 
The cubic $X$ contains no plane  and has no  associated K3 surface, see [Marq, Lemma 4.1 and Lemma 6.3]. The lattice $A(X)_{prim} $ has rank 8 and is generated by the classe $[T_i] -h^2$, where $T_i$ are cubic scrolls. Therefore $X \in \sC_{12}$. There is an isomorphism
$$t_2(S)(1) \simeq t(X)$$
\noindent as in \ref{K3iso}, see [BP, Prop.4.6].
\end{proof}
\medskip
 The family $\sF_3$ of cubic fourfolds with the anti-symplectic involution $\phi_3$ has dimension 10. Every $X \in \sF_3$ has an equation of the form
$$f(x_0,x_1,x_2) +x_0 q_0(x_3,x_4,x_5) +x_1q_1(x_3,x_4,x_5) +x_2q_2(x_3,x_4,x_5)=0,$$
\noindent with $f$ homogeneous of degree 3 and $q_0,q_1,q_2$ homogeneous of degree 2. In [Marq, Thm. 6.4] it is proved that $\sF_3 \subset \sI$ where
$$\sI =\bigcap_{d>6,d\equiv0,2 (6)}\sC_d$$
\noindent is  the intersection of all non-empty Hasset divisors. In [YY] it is proved
that $\sI$  is non- empty, contains the Fermat cubic, consists of rational  cubic fourfolds and
$$ 13 \le \dim \sI \le 16.$$
\begin{lm}\label{abelian} Every cubic fourfold  $X \in \sI$ has an associated K3 surface with a motive of abelian type.\end {lm}
\begin{proof} By the results in [ABP, Appendix A] one can choose suitable integers $d_1,\cdots,d_{19} $, with $d_1=14$, such that every $X \in \bigcap_{1\le i \le 19}\sC_{d_i}$ has an associated K3 surface whose  Picard rank  equals 19 and hence has a motive of abelian type, see [Ped]. Therefore the motive $h(X)$ is of abelian type. If  $X \in \sI $ then also $X \in\bigcap_{1\le i \le 19}\sC_{d_i}$. \end{proof}
\begin{prop}\label{$F_3$} Every cubic fourfold $X \in F_3$ is rational, has an associated K3 surface $S$ and there is an isomorphism of motives 
$$t_2(S)(1) \simeq t(X)$$
\noindent where $t_2(S)$ and $t(X)$are   of abelian type.The lattice $A(X)$ of the algebraic cohomology has rank 11.
\end{prop}
\begin{proof} Since  $F_3 \subset \sI$ every $X \in F_3$ is rational. By \ref{abelian} the cubic $X$ has an associated K3 surface whose  motive is of abelian type.\par
The plane $P: (x_0=x_1 =x_2 =0)$  is contained in $X$ and is invariant under $\phi_3$. By [Bull, Ex.3.3] There is an isnomorphism
 $$t_2(S)(1)\simeq t(X)$$
\noindent with $S$ a K3 surface. The cubic $X$ is Pfaffian,  see [Marq, Prop.6.18].
\end{proof}
\begin{rk} If $X$ is a cubic fourfold such that $\rank A(X) \ge 12$ then the transcendental lattice $T(X)$ has a primitive embedding into the K3 lattice $(E_8)^2\oplus U^3$,see [Laz ,Prop.2.9]. Therefore the maximal rank of $A(X)$ that can occur for a cubic not associated to a K3 surface, and hence conjecturally irrational, is 11. In [YY,Cor 8.18] it is proved that there are infinitely many rank 11 lattices as the algebraic cohomology of smooth cubic fourfolds which do not admit a Hodge theoretically  associated K3 surface. On the other hand the 10-dimensional family $\sF_3$ consists of rational cubic fourfolds with an associated K3 surface and whose algebraic cohomology has rank 11. \end{rk}

 \section{Automorphisms of order 3} 
A smooth cubic fourfold $X \subset \P^5$  can have prime order automorphisms of orders $2,3,5,7,11$. In [GAL,Thm.3.8] there is a list of the automorphisms of $\P^5$ of order $ 2,3,5,7,11$ and a description of the families of cubic hypersurfaces preserved by their action. 
 G.Ouchi in [Ou,Thm.8.4] proved that if a smooth cubic fourfold has a symplectic group of automorphisms $G$, whose order is different from 1 and 2, then $\rank S_G(X) \ge 12$, where 
\begin{equation} \label{notinvariant} S_G(X) =(H^4(X,\Z)^G)^{\perp} . \end{equation}
Moreover there exists a unique K3 surface $S$ such that $\sA_X\simeq D^b(S) $, with $\sA_X$ the Kuznetsov component in the derived category $D^b(X)$. Therefore $X$ has a Hodge theoretically associated K3 surface $S$, it  is conjecturally rational and there is an isomorphisms of motives $t_2(S) (1) \simeq t(X)$.\par
 In this section we consider the case of non-symplectic automorphisms of order 3 and show when  $X$ has an associated K3 surface and is rational and  when $X$  has no associated K3 surface.\par
In [GAL,Thm.3.8] there is a list of the automorphisms of $\P^5$ of order 3 and a description of the families of cubic hypersurfaces $X$ such that  the automorphism acts non-symplectically on $X$.\par
Let 
$$\sigma_1 :[x_0,x_1,x_2,x_3,x_4,x_5] \to :[x_0,x_1,x_2,x_3,x_4,\zeta x_5]; $$
be the automorphism of $\P^5$, where $\zeta$ is a primitive cubic root of 1 and let  $V_1$ be the family of cubic fourfolds invariant under $\sigma_1$. The automorphism $\sigma_1$ induces a non-symplectic automorphism on $F(X)$.The family $V_1$ has dimension 10. 
\begin{prop}\label {V_1}. Let $X$ be a general element in $V_1$. The cubic fourfold  $X$  has no Hodge theoretically  associated K3 surface. The motive $h(X)$  is finite-dimension and is not associated to the motive of a K3 surface.The Fano variety $F(X)$ is not isomorphic to the Hilbert scheme of points $S^{[2]}$ of a K3 surface $S$.\end{prop}
\begin{proof} A cubic fourfold i $X \in V_1$ has an equation of the form
$$ F(x_0,x_1,x_2,x_3,x_4) + x^3_5 =0 $$
\noindent with $ F$ homogeneous of  degree 3. The action of $\sigma_1$ is non -symplectic because $\det \sigma_1 = \zeta$ and  $\sigma(F)=F$.The cubic fourfold $X$  is a triple cover of $\P^4$ ramified along the cubic 3-fold $Y : (F(x_0,x_1,x_2,x_3,x_4)=0)$.The dimension of the family $V_1$ is 10, see [GAL]. All these fourfolds have a motive of abelian type, see [Lat 1].\par
The fixed locus of the automorphism$\sigma_1 $ on $X$ is the cubic 3-fold $Y$. The fixed locus of the automorphism induced by $\sigma_1$ on $F=F(X)$  is the surface of lines $F(Y)$, see [BCS 1,Ex. 6.4].\par
\noindent  Let $T$ be the invariant lattice of $\sigma_1$ in $H^2(F,\Z) $ and let $S =T^{\perp}$. Then
$ \rank S =22$,  and  $T = < 6> = <g,g>$, with $g$ the restriction to $F \subset \Grass (2,6)$ of the class  defining the Pl\"ucker embedding, see [BCS 1  (6.3)]. Let $P_F \to F$ be the universal family and denote by $p$ and $q$ the projection to $F$ and $X$ respectively. Then the Abel-Jacobi map

\begin{equation}\label {AJ}p_*q^* : H^4(X,\Z) \to H^2(F,\Z)\end{equation}

\noindent induces an isomorphism between the primitive cohomology  $H^4(X,\Z)_{prim} $ and $H^2(F,\Z)_{prim}$. Here $H^2(F,\Z)_{prim} =<g>^{\perp}$,   $g =p_*q^*(h^2)$, with $h$ the class of a hyperplane section and

\begin{equation} \label{divisors}   CH^1(F)_{\Q} \simeq <g> \oplus  p_*q^*(A^2(X)_{prim},\end{equation}
where
$$A^2(X)_{prim} = \{Z\in CH^2(X)_{\Q}  / [Z] \in H^4(X,\Q)_{prim}\} $$
\noindent see [SV, 21.4]. \par
 For a general element $X \in V_1$  the group $NS(F(X)) $ is equal to the invariant subgroup $T$ which has rank 1. From  the isomorphisms  in \ref{divisors} we get
$$A_2(X)_{prim} =0,$$
\noindent Therefore the algebraic lattice $A(X)$ has rank 1 and the cubic fourfold $X$  has no associated K3 surface, in the  Hodge theoretical l sense.\par
Suppose that there exists a K3 surface $S$ such that
$$t(X) \simeq t_2(S)(1),$$ 
Then
$$H^4(t(X))=H^4(X,\Q)_{tr} \simeq H^4(t_2(S)(1)) =H^2(S,\Q)_{tr}.$$
\noindent Since the dimension of $H^2(S,\Q)_{tr}$ is at most 21 we should also have $\dim H^4(X,\Q)_{tr }\le 21$, while $\dim H^4(X,\Q)_{tr} =22$, because $\rank A^2(X) =1$.\par 
If $X$ is  a general element of $V_1$ the Fano variety $F(X)$ is not isomorphic to $S^{[2]}$, with $S$ a K3 surface, because $NS(F(X)$ equals the invariant lattice that has rank 1 while $NS(S^{[2]})$ has rank at least 2.
\end{proof}
\begin{rk} (1) R.Laterveer in [Lat 2, Cor.5.2] proved that $A_0(Z)_{hom}=0$, with $Z = F(X)/<\sigma_1>$ and $X \in V_1$.\par
 (2)Since  $\rank  A(X) =1$ the family $V_1$ contains  contains no plane. The family $V_1$ contains a Pfaffian cubic fourfold $Y$ (see [BCS 2,Prop.5.1]. Since $Y \in \sC_{14}$ there is an isomorphism $F(Y) \simeq S^{[2]}$, where $S$ is a K3 surface  
 and an isomorphism $t_2(S) (1) \simeq t(Y)$. The K3 surface $S$ has a motive of abelian type.\par
(3) The family $V_1$ contains  the Klein cubic fourfold $X_K$ whose equation is of the form 
$$x^2_0x_1 +x^2_1x_2+x^2_2x_3+x^2_3x_4+x^2_4x_0 +x^3_5=0.$$
It is a triple cover of $\P^4$  branched along the Klein cubic 3-fold. The fourfold $X_K$ admits also   a symplectic automorphism  of order 11 and  therefore has an associated K3 surface. \end{rk}
Let
$$ \sigma_2:[x_0,x_1,x_2,x_3,x_4,x_5] \to:[x_0,x_1,x_2, x_3,\zeta x_4,\zeta x_5] ;$$
be the automorphism of $\P^5$,  where $\zeta$ is a primitive cubic root of 1. Let $V_2$ be the family of cubic fourfolds invariant under $\sigma_2$. The family $V_2$ has diimension 4.  Every  cubic fourfold $X\subset \P^5$ is defined by an equation of the form

\begin{equation} \label{equation} F(x_0,x_1,x_2,x_3,x_4,x_5)= f(x_0,x_1,x_2,x_3) + g(x_4,x_5) =0,\end{equation}

\noindent with $f$ and $g$ homogeneous of degree 3 .The automorphism $\sigma_2$ acts  non-symplectically  on $F(X)$.

\begin {prop} A general cubic $X \in V_2$ has an associated K3 surface $S$, belongs to $\sC_{14}$ and is rational. The motives $h(X)$ and $h(S)$ are both of abelian type.\end {prop}
\begin{proof} The automorphism $\sigma_2$ acts non-symplectically because $\det \sigma_2 =\zeta^2$ and $\sigma^*(F) =F$.\par
The fixed locus of $\sigma_2$  on $X$ is $$  P:  \{x_0= x_1= x_2=x_3=0  \   ,   \  G(x_4,x_5)=0 \} $$
consisting of 3 points $p_1,p_2,p_3$ and the cubic surface
$$S  :  \{x_4 =x_5 =0  \  ,  \   f (x_0,x_1,x_2,x_3)=0 \}. $$
The 3 points  $p_i$ lie on the line $x_0=x_1=x_2=x_3=0$ which is not contained in $X$. Let $l$ be a line joining  a point $p_i$ and a point $Q \in S$. The line $l$ is invariant and is contained in $X$, because otherwise the third point $T$  of intersection in $l \cap X$  would be a fixed point of $\sigma$ while $T \ne p_i, T \notin S$. Therefore on $F(X)$  the automorphism induced by $\sigma_2$ 
has three fixed surfaces $K_1,K_2,K_3$ each one isomorphic to the rational cubic $S$. Moreover each fixed line on $S$ determines a fixed point on $F(X)$ so we get  27 isolated fixed points.\par 
Let $K_1=F(X)_{p_1}$ and  $ K_2=  F(X)_{p_2}$  be the two cubic surfaces in $F(X)$ which are the locus of lines joining the point $p_1$ and the point $p_2$  with  the points on $S$. There are 27 planes contained in $X$ and passing trough $p_1$ i.e.  the planes $P^{(i)}_1= \P^2_{<l_i,p_1>}$, with $(l_1,\cdots, l_{27})$ the lines on $S$. Similarly there are 27 planes contained in $X$ and passing trough $p_2$,  i.e.  the planes $P^{(i)}_2=\P^2_{<l_i,p_2>}$. The planes  $P^{(i)}_1$ and  $P^{(i)}_2$ meet along the line $l_i$ while two planes $P^{(i)}_1$ and $P^{(j}_2$, with $i \ne j$ , $i,j =1\cdots 27$,  meet in one point  $r_{ij}$ if  $l_i \cap l_j \ne \emptyset$. Since  a smooth cubic fourfold has at most 10 planes with one point intersection, see [DM, Cor.7.3] there  are some of the planes  $P^{(i)}_1, P^{(j)}_2$ which are contained in $X$ and  do not intersect. Every cubic fourfold $X$ containing  two disjoint planes  $P,P'$ belongs to $\sC_8 \cap \sC_{14}$, hence it  is rational and has an associated K3 surface $S$, see [Hass ,1.2].\par 
The invariant lattice $T_{\sigma_2}(X) \subset H^2(F(X) ,\Z)$ has rank 13, see [BCS 1, Ex.6.5], and hence the isomorphism in \ref{divisors} implies  $\rank A(X) \ge 13 $.\par 
Let $X_1,X_2$  be the cubics of dimension 3 and 1, respectively
$$X_1 :f(x_0,x_1,x_2,x_3)  - y^3 =0  \   ,  \ X_1 \subset \P^4 $$
$$X_2 :  g(x_4,x_5) + z^3 =0   \   ,   \  X_2 \subset \P^2$$
The rational map  $ \P^4 \times \P^2 \dashrightarrow \P^5 $
$$\Phi  :[x_0,x_1,x_2,x_3,y; x_4,x_5,z] \to [x_0/y ,x_1/y, x_2/y, x_3/y ,x_4/z,x_5/z]$$
induces a dominant rational map  $\Psi$ of degree 3 of $X_1\times X_2$ on the hypersurface in $\P^5 : (z_0,z_1,z_2,z_3,z_4,z_5),$ defined by the equation
$$f(z_0,z_1,z_2,z_3) +g(z_4,z_5) =0$$
\noindent which is isomorphic to $X$, see [CT, Prop.1.2] and [Lat 3,Prop.6].\par
The  indeterminacy locus  of $\Psi$ is
$$ Y = (X_1 \cap (y=0) \times (X_2\cap (z=0) .$$
Blowing up $X_1 \times X_2 $ along $Y$ gives a finite surjective morphism  $ \overline{X_1\times X_2)}\to X$, with $ \overline{X_1\times X_2} =\Bl_Y(X_1 \times X_2)$. Therefore $h(X)$ is a direct summand of $h( \overline{X_1\times X_2})$ and 
$$h (\overline{X_1\times X_2} ) \simeq h(X_1 \times X_2) \oplus h(Y)(1).$$
Since  $ (X_1 \cap (y=0)$ is a cubic surface and $ (X_2\cap (z=0)$ consists of 3 points the motive $h(Y)$  has no transcendental part. Therefore the transcendental motive $t(X)$ is a direct summand of $h(X_1 \times X_2)$ where
$$h(X_1 \times X_2) = \sum _{ 0 \le i \le 8}h_i(X_1 \times X_2 )  \   ;   \   h_i(X_1 \times X_2) =\sum_{p+q=i}h_p(X_1) \times h_q(X_2) .$$
The motive of the cubic 3-fold $X_1$  has a Chow-K\"unneth decomposition
$$h(X_1) \simeq \un \oplus \L \oplus N \oplus \L^2\oplus \L^3,$$
\noindent where  $N =h_1(J)(1)$ with $J$ an abelian variety isogenous to the intermediate Jacobian  $J^2(X_1)$, see[GG]. The motive $h(X_2)$ of the elliptic curve $X_2$ is isomorphic to $\un \oplus h_1(X_2) \oplus \L $ and is of abelian type. Therefore 
the transcendental motive $t(X)$ is a direct summand of $(h_1(J) \otimes h_1(X_2) (2)$ and is of abelian type. Since $X \in \C_{14}$, there is an isomorphism of motives
$$t_2(S) (1) \simeq t(X)$$
\noindent  and therefore also  the K3 surface $S$ has a motive of abelian type.
\end{proof}
\begin{prop} \label{V_3}Let 
$$\sigma_3: [x_0,x_1,x_2,x_3,x_4,x_5] \to :[x_0,x_1,x_2,\zeta x_3,\zeta x_4,\zeta^2 x_5]; $$ 
be the automorphism of $\P^5$, where $\zeta$ is a primitive cubic root of 1 and let $V_3$ be the family of cubic  fourfolds  with a non-symplectically action of the automorphism $\sigma_3$.The family $V_3$  has dimension 7. For a general element $X \in V_3$ the lattice $A(X)$ has rank  7  and it is generated by the classes of 27 cubic scrolls. The fourfold $X$  belongs to the Hasset divisor $\sC_{12} $, contains no plane and has no  Hodge theoretically associated K3 surface. There is an isomorphism of motives in $\sM_{rat}(\C)$
$$t_2(S) (1)\simeq t(X)$$
\noindent where $S$ is  a K3 surface.\end {prop} 
\begin{proof} Every $X \in V_3$ is defined by an equation of the form
$$F(x_o,x_1,x_2,x_3,x_4,x_5)=  f(x_0,x_1,x_2) +g(x_3,x_4) +x^3_5 +$$
$$+x_3x_5 l_1(x_0,x_1,x_2) +x_4x_5 l_2(x_0,x_1,x_2) =0,$$
\noindent where $f,g$ are homogeneous of degree 3 and $l_1,l_2$ are linear forms. The automorphism $\sigma_3$ is non- symplectic because $\det \sigma_3 =\zeta $ and $\sigma_3(F) =F$.The family $V_3$ has dimension 7, see [GAL,Thm.3.8].\par
In [BG, Lemma 4.9 and Prop.4.11] it is proved that  $X \notin \sC_d$, for  $d\equiv 2 (6)$. Therefore $X $ contains no plane. The fourfold  $X$ contains 27 families of cubic scrolls $\{[T_i], [T^*_i] \}$, $i = 1,\cdots 27$, such that $A(X)$ is generated by $ [T_i]$.Therefore $X \in \sC_{12}$. Since $d =12$ satisfies the numerical condition in (**) there is an isomorphism of motives as in \ref{motivic}. \par
There is no  primitive embedding of $T(X)(-1)$ in a K3 lattice, hence  there is no associated K3 surface, see [BG,lemma 4.7]. \par
\end{proof}
\begin{rk}The family $V_3$ contains the cubic fourfold $Y$ defined by the equation 
$$Y:  f(x_0,x_1,x_2) +g(x_3,x_4) +x^3_5=0$$
\noindent which is  also invariant under the symplectic automorphism 
$$\tau:[x_0,x_1,x_2,x_3,x_4,x_5] \to [x_0,x_1,x_2,\zeta x_3,\zeta x_4,\zeta x_5].$$
The cubic fourfold $Y$ is rational and has a motive of abelian type, see [CT] and [Lat 3]. The fixed locus $\sigma_3$ on $F(X)$, for a general  $X \in V_3$, is parametrized by $C \times D$,
were $C,D$ are the elliptic curves defined by
$$C: f(x_0,x_1,x_2) =0  \   in  \  \P^2: x_3=x_4=x_5=0,$$
$$D : g(x_3,x_4) +x^5=0 \ in   \   \P^2: x_0=x_1=x_2=0$$ 
\noindent and coincides with the fixed locus for the action of $\tau$ on $F(Y)$. In [Kaw,Prop.3.2] it is proved that the quotient   variety  $F(Y)/<\tau>$ has a crepant resolution $\tilde Y \to F(Y)/<\tau>$, with $\tilde Y$ birational  to  the
generalized  Kummer variety $K^2(C \times D)$. The same argument shows that the  motive of the quotient $F(X)/<\sigma_3>$ is a direct summand of $K^2(C \times D)$, hence it is of abelian type.

\end{rk}

\medskip

Let
$$\sigma_4 : [x_0,x_1,x_2,x_3,x_4,x_5] \to :[x_0,x_1,\zeta x_2,\zeta x_3,\zeta^2x_4, \zeta^2 x_5] .$$
be the automorphism  of $\P^5$, where $\zeta$ is a primitive cubic root of 1. There are two families of cubic fourfolds  which are invariant under the action of $\sigma_4$, a first  one   where $\sigma_4 $ acts non-symplectically and a 
second one where $\sigma _4$ acts  symplectically, see [GAL,Thm. 3.8]. Let $V_4$ be the family of cubic fourfolds with a non-symplectical action of $\sigma_4$. Every $X \in X$ has an equation of the form
$$F(x_0,x_1,x_2,x_3,x_4,x_5) = x_2 (L_2(x_0,x_1) +x_3 M_2(x_0,x_1) +x^2_4L_1(x_0,x_1)+$$
$$+x_4x_5 M_1(x_0,x_1) +x^2_5 N_1(x_0,x_1)+ x_4N_2(x_2,x_3) +x_5P_2(x_2,x_3) =0$$
\noindent where $L_2,M_2,N_2, P_2$ are homogeneous polynomial of degree 2 and $L_1,M_1,N_1$ are linear factors. The automorphism $\sigma_4$ acts non-symplectically on $F(X)$ because $\det \sigma_4 =1$ and 
$\sigma^*_4(F) =\zeta$.
\begin{prop}\label {V_4} The family $V_4$ has dimension 6. A general fourfold $X \in V_4$ contains nine disjoint planes $F_1,\cdots, F_9$ and  $A(X)$  has a basis given by
$\{h^2,[F_1],\cdots,[F_8]\}$. Therefore $X \in \sC_{14}$, is rational and has an associated K3 surface.\end{prop}
\begin{proof}  A plane in $\P^5$ that is invariant for the action of $\sigma_4$ has an equation of the form
$$ \{a x_0=b x_1, cx_2=dx_3,e x_4=f x_5 \}.$$
There are  exactly nine  disjoint planes $F_1,\cdots,F_9$ that are invariant  under $\sigma_4$ and are contained in $X$. Moreover 
$$h^2 = (1/3) \sum_{1 \le i \le 9}[F_i],$$
\noindent  see [BG, Prop. 4.13].
\end{proof}

\section{Singular K3 surfaces}
In this section we describe  families of special cubic fourfolds  $X$ with an associated K3 surface $S$ whose  Picard rank equals to 20, i.e.  a singular K3 surface.  These K3 surfaces have a motive of abelian type, see [Ped]. Since the motive $h(X)$ is associated to the motive of  $S$, in the sense of \ref{motivic}, also  $h(X)$ is of abelian type. \par
The following result appears in [ABP ,Thm.1.1] :  let  $3\le n \le 20$ and let $d_k \geq 8, d_k \equiv 0,2[6]$, where 
$$d_3,..,d_n= 6 \displaystyle{\prod_i} p_i^2  \   or  \ 6 \displaystyle{\prod_i} p_i^2+2,$$ 
\noindent with $p_i$ a prime number. Then
 
\begin{equation}\label{family}
     \displaystyle{\bigcap_{k=1}^{n}} \mathcal{C}_{d_k} \ne \emptyset,
\end{equation}
 
Note that the first two discriminants $d_1$ and $d_2$  may  be chosen completely arbitrarily.  A very general cubic fourfold $X$ has $CH_2(X)=\mathbb{Z}$, and each time that we take the intersection with a  divisor $\mathcal{C}_{d_i}$ we know that we are adding an algebraic cycle inside $A(X)$. Therefore the generic cubic fourfold in $\displaystyle{\bigcap_{k=1}^{n}} \C_{d_k} \ne \emptyset$ has $\rank(CH_2(X))=n+1$. If one of the discriminants $d_i$ defines a divisor of cubics with associated K3 surface, it means that all the cubic fourfolds inside the intersection \ref{family} have associated K3 surfaces and, by the definition itself, the generic associated K3 surface has Picard  rank equal to $n$.\par
Let   $\displaystyle{\bigcap_{k=1}^{20}} \C_{d_k} \ne \emptyset$, with at least one of the divisor $\sC_{d_k}$ parametrizing cubic fourfolds with an associate K3 surface. For every $X \in \displaystyle{\bigcap_{k=1}^{20}} \C_{d_k}$  the lattice $A(X)$ of algebraic cycles has rank 21, hence the primitive cohomology of the associated K3 surface has rank 19 and the Neron-Severi group has rank 20.\par
The singular K3 surface $S$  associated to  a cubic fourfold $X\in \displaystyle{\bigcap_{k=1}^{20}} \C_{d_k}$ admits a Shioda-Inose structure. In particular there is a birational map from $S$ to a Kummer surface, which is the quotient of an Abelian variety by an involution.\par 
Therefore the motive $h(X) $ is associated to $ h(S) $ and are both of abelian type. The Fano variety $F(X)$ is birational to the Hilbert scheme of length 2 subschemes of a Kummer surface.\par
\medskip

A different way to construct cubic fourfolds with associated singular K3 surfaces is to consider a finite group $G$ of symplectic automorphisms acting on $X$ such that the covariant lattice $S_G(X)$ has maximal rank. The moduli space parametrizing cubic fourfolds
with a symplectic group $G$  has dimension  $20 -\rank S_G$. The  case  $\rank S_G =20$ corresponds to cubic fourfolds $X$ such that the lattice of algebraic cycles $A(X) =<S_G,h^2>$ , with $h$ the class of hyperplane section, has rank 21. All these cubics have an associated K3 surface $S$ with Picard rank equal to 20. Therefore all these cubic fourfolds have a motive of abelian type.\par
In [LZ,Thm.1.8] a list  is given of all the groups $G$ which are the symplectic automorphism group of a cubic fourfold $X$ with $\rank S_G(X) =20$. There  are six of such groups and 8 associated cubic fourfolds. The list contains the Fermat cubic $X_F$, the Klein cubic fourfold
$X_K$ and the Clebsch diagonal cubic $X_C$. The Fermat cubic  $X_F$ is the only smooth cubic fourfold with a symplectic automorphism of order 9. It is rational and is  contained in the intersection of all Hasset divisors $\sC_d$, with $d $ satisfying (*). The Klein
cubic fourfold $X_K$ is the only cubic with a symplectic automorphism of order 11. It is contained  in all  $\sC_{3d}$  with $d$ even and $d >4$.\par
The diagonal cubic $X_C$ has equation 
$$F(x_0,x_1,x_2,x_3,x_4,x_5) = x^3_0+x^3_1+x^3_2+x^3_3+x^3_4+x^3_5 -(x_0+x_1+x_2+x_3+x_4+x_5)^3=0.$$
\noindent and has a symplectic automorphism of order 7.It is i also invariant with respect to a anti-symplectic involution, see [LZ,Cor.6.9]. Therefore $X_C$ contains an Eckardt points and  belongs to $\sC_8 \cap \sC_{12}$. \par

\section {K3 surfaces with a symplectic automorphism}
In the previous sections we described several families of cubic fourfolds $X$, with a finite group of automorphisms, such that the motive $ h(X)$ is associated to the motive of a K3 surface $S$. Here we consider  the case of K3 surface $S$ such that 
$ S^{[2]} \simeq F(X)$ and show that $t_2(S)(1) \simeq t(X)$. If  $S$ is equipped with a symplectic automorphism of prime order $p$ then also $X$ has a symplectic automorphism of order $p$. In the case $p=3$ we give examples of K3 surfaces  $S$ whose motive is associated to different cubic fourfolds $X_1 \in \sC_{d_1} $ and  $X_2 \in \sC_{d_2}$ with $d_1 \ne d_2$.\par
Let $\sF_g$  be  the  moduli space of smooth polarized K3 surfaces $(S,L)$ of  genus  $g$ and degree $d$, with $L^2 =d = 2g -2$ and  $g = n^2 +n +2$. There is a  rational map of degree 2
\begin{equation} \label{moduli} \phi :  \sF_g \dashrightarrow \sC_d, \end{equation}
\noindent where $d =2(n^2+n+1)$, sending $(S,L)$ to the class  $\vert X\vert $ of a cubic fourfold $X \in \sC_d$. There is an isomorphism
\begin{equation} \label {Fano} S^{[2]} \simeq F(X).\end{equation} 
 \begin{lm}\label{iso}  Let $X$ be  a cubic fourfold and let $F =F(X)$ be the Fano variety of lines. Suppose that $F \simeq S^{[2]}$, with  $S$ a K3 surface.  Then there is an isomorphism of motives $ t_2(S)(1)\simeq  t(X)$ in $\sM_{rat}(\C)$. 
\end{lm}
\begin{proof} In the incidence diagram
 \begin{equation}\label{incidence} \CD      P_F@>{q}>>  X\\
@V{p}VV      @.   \\
 F (X)  \endCD           \end{equation}
\noindent  the universal line $P_F$,  as a correspondence in $A_5(F\times X)$, gives a map 
$$P_* : h(F)(1) \to h(X)$$
 By the results in [deC-M,Thm.6.2 ] $h(S)$ is a direct summand of $h(S^{[2]})\simeq h(F)$. Therefore we get a map  
$$ \CD  h(S)(1)@>>>h(F)(1) @>{P_*}>>h(X) \\ \endCD$$
  Let 
$$h(S) \simeq \un \oplus \L^{\oplus \rho(S)} \oplus t_2(S) \oplus \L^2$$
be a refined Chow-K\"unneth decomposition, as in [KMP]. By composing with the inclusion $t_2(S)(1) \to h(S)(1)$ and the surjection $h(X) \to t(X)$ we get  a map of motives  in  $\sM_{\rat }(\C)$,
$$f :  t_2(S)(1) \to t(X)$$ 
For two distinct points $x,y \in S$ let us denote by $[x,y] \in S^{[2]}\simeq F $ the point of $F$ that corresponds to the subscheme $x \cup y \subset S$. If  $x=y$  then $[x,x]$ denotes the element in $A^4(F)$ represented by any point corresponding to a non reduced subscheme of length 2 on $S$ supported on $x$. With these notations the special degree 1 cycle $c_F\in A^4(F)$ (see [SV,Lemma A3]), given by any point on a rational surface $W \subset F$, is represented by the point $[c_S, c_S]  \in F$, where $c_S$  is the Beauville-Voisin cycle in $A_0(S)$ such that $c_2(S) = 24  c_S$. We also have (see [SV,Prop. 15.6]) 
$$(A_0(F)_2 )_{hom} = <[c_S,x]- [c_S,y]>,$$
\noindent where $A_0(F) = A_0(F)_0 \oplus A_0(F)_2 \oplus A_0(F)_4$ is the Fourier decomposition, as in [SV]. \par
 We claim  that the map $\phi : A_0(S) \to A_0(S^{[2]})=A_0(F)$ sending $[x]$ to $[c_S, x]$ is injective and hence the surjective map $A_0(S)_0 \to (A_0(F)_2 )_{hom} $, sending $ [x] -[y]$ to $[c_S,x]- [c_S,y]$ gives an isomorphism
 $$A_0(S)_0  \simeq (A^4(F)_2)_{hom}.$$
  \noindent The variety $S^{[2]}$ is the blow-up of the symmetric product $S^{(2)}$ along the diagonal $\Delta \cong S$. Let $\tilde S$ be the inverse image of $\Delta$ in $S^{[2]}$. Then $\tilde S$ is the image of the closed embedding $s \to [c_S, s]$. By a result proved in [Ba,Thm.2.1] the induced map of 0-cycles $A_0( \tilde S) \to A_0(S^{[2]})$ is injective. Therefore the map $\phi$  is injective.\par
 \noindent From the isomorphisms $A_0(S)_0 \simeq (A_0(F)_2)_{hom}$  and $A_0(F)_2)_{hom}\simeq A_1(X)_{hom}$ (see [BP]) we get
 $$A^3(t_2(S)(1) = A^2(t_2(S))=A_0(S)_0 \simeq A^3(X) _{hom}\simeq A^3(t(X))$$
 Since $A^i(t_2(S)(1)) =A^i(t(X))=0$ for $i \ne 3$ the map $f  : t_2(S)(1) \to t(X)$ gives an isomorphism on all Chow group, see [BP,Lemma 2.3]. Therefore  $t_2(S)(1) \simeq t(X)$ 
    \end{proof}
Let $X \in \P^5_{\C}$ be a smooth cubic fourfold and let $F = F(X)$ be its Fano variety of lines. The irreducible holomorphic symplectic variety $F$ is deformation equivalent to the Hilbert scheme $S^{[2]}$ of two points on a K3 surface $S$.
A   finite order automorphism $\sigma$  of $S$ induces an automorphism $\sigma^{[2]}$ on the Hilbert scheme  $S^{[2]}$ that is called {\it natural}.
\begin{defn} An automorphism $\sigma$   of $F(X)$ is called  {\it standard} if $(F(X), \tilde \sigma) $ is deformation equivalent  to
$(S^{[2]},\phi^{[2]})$, where $\phi$ is an automorphism of $S$. \end{defn}
\begin{ex} (1) The  Klein cubic fourfold $X$, whose equation is of the form 
$$x^2_0x_1 +x^2_1x_2+x^2_2x_3+x^2_3x_4+x^2_4x_0 +x^3_5=0,$$
\noindent admits a symplectic automorphism of order 11 which cannot be standard because there is no symplectic automorphism of order 11 on a K3 surface.\par
 (2) The  non-symplectic automorphisms of order 3 on the cubic fourfolds belonging to the families $V_1$ and $V_3$ in Sect.3 are non-standard, see [BCS 1,Cor.7.6]\ \end{ex}
\begin{defn} Let $X\in \P^5_{\C}$ be a smooth cubic fourfold and let $F(X)$ be its Fano variety of lines equipped with the polarization $g$ which is by definition the  class in $H^2(F,\Z)$ of the  restriction to $F$ of the Plucker line bundle $\sL$  on the Grassmannian $G(2,6)$ of lines in $\P^5$. An automorphism $\psi$ of $F(X) $ is {\it polarized} if it preserves the Plucker  polarization, i.e.  $\psi^*g = g$.\end{defn} 
An automorphism of $F(X)$ is  polarized if and only if it is  induced from an automorphism of the cubic fourfold $X$, see [Fu,Lemma1.2]. \par
 \begin{lm}\label{polarized} Let $(S,L)$ be  a polarized K3 surface of genus $g$ and degree $d$, with  $L^2 =2g-2=d$ and  $d =2(n^2+n+1)$. Let $X \in \sC_d$ be the corresponding cubic fourfold in \ref{moduli}.  Let $\sigma$ be a symplectic automorphism of $S$  of prime order $p$ and let $\sigma^{[2]}$ be the induced automorphism on $S^{[2]} \simeq F(X)$. The  natural automorphism  $\sigma^{[2]}$ on $F(X) $ is polarized and hence it is induced  by an  automorphism  $\tau$  on $X$. The fixed locus of the automorphism $\tau$ on $X$ consists of a finite number of isolated points.\end{lm}
\begin{proof}If $\sigma$ is  a symplectic automorphism of $S$ of  prime order $p$, then $p =3,5,7$. There is an isomorphism
$$H^2(F(X), \Z) \simeq H^2 (S^{[2]},\Z) \simeq H^2(S,\Z) \oplus \Z \delta,$$
 Here $\delta = (1/2 )|\Delta | \in H^2(S^{[2]},\Z)$, with $\Delta \subset S^{[2]} $ the divisor consisting of  zero- dimensional closed subschemes supported only at a single point. Then
 
\begin{equation} \Delta =  \P(T_S) = \bigcup_{p \in S} \Delta_p, \end{equation}
\noindent where $\P(T_S)$ is the geometric projectivization of the tangent bundle $T_S$.  For each $p \in S$  $\Delta_p$ is  the rational curve consisting of 0-dimensional subschemes $\xi \in \Delta$ such that $\supp (\xi) =p$, see [FV,2.1].\par
The Fano variety $F(X)$ has a polarization  $g$, where
$$g =2 f - (2n+1) \delta,$$
\noindent  with  $ f = \vert L \vert \in H^2(S,\Z)$, see [Hass, Thm.29]. The symplectic automorphism $\sigma$  of the polarized K3 surface $S$ fixes the classes $f$ and   the automorphism $\sigma^{[2]}$  of  $S^{[2]} $ fixes the class of $\Delta$. Therefore  the induced automorphism fixes the polarization $g$. Therefore the  natural automorphism $\sigma^{[2]}$ on $F(X)$ is polarized and hence it is induced by  an automorphism $\tau$ of $X$. \par
 The automorphism $\sigma^{[2]}$ has a finite number $m_p$  of fixed points, with $m_3 =27, m_5 = 14,m_7 =9$, see [Bois, Ex.2]. Every subspace $\xi \in S^{[2]}$  fixed under  $\sigma^{[2]}$ is the union of reduced subspaces with disjoint supports of the form\par 
(1) $\{p \}$ with $\sigma(p)=p$ ,\par
(2)$ \xi_p$ with $\xi_p$ a non reduced subscheme of length 2 supported on $p$ with $\sigma(p)=p$,  see [Bois, 4.2 and Rk. 6].\par
The fixed locus  of $\sigma$ on $S$ consists on a finite number of isolated points and and the fixed 1-dimensional subschemes $\xi_p$ in (2) correspond to $\Delta_p \subset S^{[2]}$, with $\sigma(p) =p$.\par
 Let  $x \in X$ and let $\tau$ be the  automorphism induced on $X$. Let $\tau(x) =x$  and let $l \in F(X) $ be a line trough $x$. Then  also  the line $ \sigma^{[2]}(l)$ passes trough $x$. Therefore the 1-dimensional subscheme $C_x\subset F(X)$ of all lines trough $x$
corresponds, under the isomorphism $F(X) \simeq S^{[2]}$, to a 1-dimensional subscheme  fixed under $\sigma^{[2]}$. So there  are only a finite number of points on $X$ fixed under $ \tau$. 
\end{proof}
\subsection{The case p=3}  Let $S$ be a projective K3 surface admitting a symplectic automorphism $\sigma$  of order 3 and let $Y $ be the K3 surface which is a  minimal desingularization of the quotient $S/\sigma$. The surface$S$ has six isolated fixed points $P_1,\cdots,P_6$ and the surface $S/\sigma$ has 6 singularities of type $A_2 $. The surface $Y$ contains 12 irreducible curves, i.e. 6  disjoint pairs of rational curves meeting in a point. In the diagram
\begin{equation} \label{quotient} \CD
\tilde S@>{	\alpha}>>  S    \\
@VVV                  @V{\pi}VV  \\
Y@>{\beta}>>S/\sigma    \endCD \end{equation}
\noindent the surface $\tilde S$ is the blow-up of $S$ at the six  fixed points $P_i$ of $\sigma$. \par
The family $\sS$ of polarized K3 surface with a symplectic automorphism of order 3 is the union of countably many components of dimension 7. Similarly the family $\sT$ of  K3 surfaces that are 
quotients of K3 surfaces with a symplectic automorphism of order 3 is the union of countably many components of dimension 7. The correspondence between K3 surfaces in $\sS$ and in $\sT$ has been described in [GM].
A  K3  surface $S$ is a generic member of the component $\sS_{2d}\subset \sS$ of those surfaces having  a polarization $L$ of degree $2d $ if and only if the corresponding K3 surface $Y$ is a generic member of the component $\sT_{6d} \subset \sT$ of those surfaces having a polarization $H$ of degree $6d$. Conversely a K3 surface $Y \in \sT_{2e} \subset \sT$  with a polarization of degree $2e$ corresponds to a K3 surface $S \in \sS_{6e}\subset \sS$ with a polarization of degree $6e$, see see [GM,Cor.5.3].\par
Let $S \in \sF_g$, where  $g =n^2 +n +2$ with   a symplectic automorphism $\sigma$  of order 3 and let $X \in \sC_d$ be the image of $S$ under the rational map in \ref{moduli}. Then $F(X) \simeq S^{[2]}$ and $\sigma$ induces a natural automorphism $\sigma^{[2]}$ on $F(X)$ with 27 isolated fixed points. The automorphism $\sigma^{[2]}$ is induced by an automorphism $\tau$ of $X$, which in turn is induced by a linear automorphism of $\P^5$. Then one can choose coordinates $[x_0,x_1,x_2,x_3,x_4,x_5]$ in such a way that $(X,\tau)$ belongs to one of the families of cubic fourfolds with a symplectic automorphism of order 3, as described in [Fu, Thm.1.1]. There are two families  of cubic fourfolds with a symplectic automorphism of order 3, whose fixed locus on $F(X)$ consists of 27 isolated fixed points.\par
The  first family $F_3$ has dimension  8. Every cubic in $F_3$ is invariant under the automorphism of $\P^5$ 
\begin{equation} \label{auto} \tau_1:  [x_0,x_1,x_2,x_3,x_4,x_5] \to [x_0,x_1,x_2,x_3,\zeta x_4, \zeta^2 x_5]\end{equation}
\noindent with $\zeta^3 =1$ and has an equation of the form
$$F(x_0,x_1,x_2,x_3,x_4,x_5) =  f(x_0,x_1,x_2,x_3) +x^3_4 +x^3_5+ x_4x_5 L(x_0,x_1,x_2,x_3)=0,$$
with $f$ of degree 3 and $L$ of degree 1. The fixed locus of the automorphism  $\tau_1$ on $X$ is the cubic surface $S :  f(x_0,x_1,x_2,x_3) =0$. Since this  fixed locus is not finite the automorphism  induced  by \ref{auto} on $F(X)$ cannot coincide with $\sigma^{[2]}$ of $F(X)$, hence it is not natural. In [Mong,7.2.8] it is proved that the  order 3 symplectic automorphism of the cubics in $F_3$ are standard.\par
The  second family $G_3$  has dimension 8 . Every cubic in $G_3$ is invariant under the automorphism of $\P^5$ 
$$ \tau_2 :  [x_0,x_1,x_2,x_3,x_4,x_5] \to [x_0,x_1,\zeta x_2,\zeta x_3,\zeta^2 x_4, \zeta^2 x_5] $$
and has an equation of the form
$$F(x_0,x_1,x_2,x_3,x_4,x_5) = f_1(x_0,x_1) +f_2(x_2,x_3) +f_3(x_4,x_5) + \sum_{i,j,k} a_{ijk}x_i x_j x_k$$
\noindent where $f_1,f_2,f_3$  have degree 3 and $i =0,1; j= 2,3 ; k= 4,5$. The  fixed locus  of $\tau_2$ on $X$ is finite and   consists of the 9 isolated points $P_{i,1},P_{i,2}, P_{i,3}$ satisfying the the equation $f_i=0$, for $i =1,2,3$ .
The fixed locus on $F(X)$ corresponds to the 27 lines $\overline {P_{i,k}P_{j,l}}$, for $1 \le i<j \le 3$ and $k,l=1,2,3$, see [Fu,IV-(5)].
 Each point on $X$ fixed by the automorphism is the intersection of 3 lines in $F(X) $ fixed by the automorphism.\par
Since the moduli space of polarized K3 surfaces with a symplectic automorphism of order 3 has dimension 7 (see [GS],Cor.5.1]), the image in $ G_3$  has codimension 1 and contains the cubics $X$ such that the automorphism $\tau_2$ is natural.\par
Let $S$ be K3 surface  with a polarization  $L$ of degree $L^2= 2d =2g -2 $, where $d=n^2 +n +1$, with a symplectic automorphism $\sigma$ of order 3. Then  a desingularization $Y$ of $S/\sigma$ has a polarization $H$ of degree $6d$ and genus 
$g' =3n^2+3n +4$. Similarly if  $S$ is  a K3 surface  with a polarization  $L$ of degree $L^2= 6e =2g -2 $, where $g =m^2 +m +2 $, then the desingularization $Y$ of  $S/\sigma$ has a polarization $H$ of degree $2e$  and genus 
$g'=( m^2+m+1)/3 +1$. For  $ m=4$ we get  $g= 22$ and $g' = 8$. Therefore $S \in \sF_{22} $ and  $Y \in \sF_8$ \par
\begin{ex}Let $ S$ be a   polarized K3 surface in $\sF_{22}$, with a symplectic automorphism $\sigma$ of order 3 and let $\tilde X \in \sC_{42}$ be the cubic fourfold corresponding to $S$ under the map in \ref{moduli}. There is an isomorphism $F(\tilde X) \simeq S^{[2]}$ and an order 3  automorphism $\tau$ of $\tilde X$ induced by the automorphism $\sigma^{[2]}$ of $F(\tilde X)$.  Let $Y$ be a desingularization of $ S/\sigma$ and let $X \in \sC_{14}$ be the cubic fourfold corresponding to $Y$ under the map $ \sF_8  \dashrightarrow  \sC_{14}$.Then 
\begin{equation} \label{Fano} F(\tilde X) \simeq  S^{[2]}   \   ,    \    F(X) \simeq Y^{[2]}. \end{equation}
The  8-dimensional family $G_3$ of cubic fourfolds with a symplectic automorphism of order 3 contains two divisors , one corresponding to the cubics $\tilde X$ in $\sC_{42} $ and one to the Pfaffian cubics $X \in \sC_{14}$.\par 
Since $42 \equiv 0 (6)$ if  $S$ and $ S'$ have the same image in $\sC_{42}$, under the rational map $\sF_{22} \dashrightarrow \C_{42}$, then $ S$ and $ S'$ are Fourier- Mukai partners, see [Brak,Thm.1.1]. Therefore $D^b( S) \simeq D^b(  S')$ 
and the Chow motives $h( S)$ and $h(  S')$ are isomorphic.\par
 The quotient map $  S  \to  S/\sigma$ and the resolution $Y \to \ S/\sigma$  give a generically 3:1 rational map $\pi :  S \dashrightarrow  Y$ and therefore a rational map
$$\pi^{[2]} : S^{[2]} \dashrightarrow Y^{[2]},$$
\noindent  which is defined outside the 27 fixed points on $ S^{[2]}$. The rational map  $\pi^{[2]}$
induces, via the isomorphisms in \ref{Fano}, a rational map  $F(\tilde X) \dashrightarrow F(X)$.\par
 Since the transcendental motive $t_2(-)$ of a surface is a birational invariant the maps $\tilde S \to  S$   and $\tilde S \to Y$    in \ref{quotient} induce a map $\theta : t_2(S )\to t_2(Y)$ that is the projection onto
 a direct summand, see [Ped,Prop.1]. Therefore 
$$  t_2(S)=t_2(Y) \oplus N.$$
By a result of Huybrechts a symplectic automorphism acts trivially on the Chow group of 0-cycles $A_0(S) =CH_0(S) \otimes \Q$. Since $A_0(t_2(S) =A_0(S)_{hom} $ , $A_i(t_2(S))=0$, for $ i \ne 0$, and $A_0(S)^{\sigma^*} =A_0(S)$ we get $A_i(N)=0$, for all $ i$. Therefore
$N=0$,and  hence 
$$t_2(S) =t_2(Y)$$
The cubic fourfolds $\tilde X$ and $X$ are both rational.The transcendental motives $t(\tilde X)$ and $t(X) $ are isomorphic because  $t_2(  S) (1) \simeq t( \tilde X)$ and $t_2(Y) (1) \simeq t(X)$ (see \ref{iso}), with  $t_2(S) \simeq t_2(Y)$ . Therefore
\begin{equation} t_2(S) (1) \simeq t(\tilde X) \simeq t(X).\end{equation}
The K3 surface $S$ is associated to the motive of $X$ and of $\tilde X$ in the sense of \ref{motivic}.\par
Let $\sA_{\tilde X}$ and $\sA_X$ be the  Kuznetsov components of $ \tilde X$ and $X$, respectively. Since $X$ and $\tilde X$ belong to  a Hassett divisors $\sC_d$, with d satisfying the numerical condition in (*), there are isomorphisms 
$$\sA_{\tilde X} \simeq D^b(S)   \   ;   \   \sA_X \simeq D^b(Y),$$
\noindent where
$$D^b(\tilde X) = <\sA_{\tilde X}, \sO_{\tilde X}, \sO_{\tilde X}(1), \sO_{\tilde X}(2)>.$$
\noindent and similarly for a semi-orthogonal decomposition of $D^b(X)$. Let $G= <\sigma>$ and let $D^b_G(S)$ be the derived category of $G$-equivariant coherent sheaves on $S$. Then
$$ D^b(Y) \simeq D^b_G(S),$$
\noindent see [XH, 3.1]. 
\end{ex}


\begin{thebibliography}{10} 

 \bibitem[ABP]{ABP} H.Awada, M.Bolognesi and C.Pedrini.{\it A family of special cubic fourfolds with motive of abelian type}, arXiv:2007.07193v1 [math.AG], 14 Jul 2020.

\bibitem [AS]{AS} M.Artebani and A.Sarti, {\it Non-symplectic automorphisms of order 3 on K3 surfaces}, Math.  Ann. 342,  903-921 (2008)., 

\bibitem [BG]{BG}S.Billi and A.Grassi {\it Non-symplectic automorphism of prime order of O'Grady's tenfolds and cubic fourfolds}, arXiv:2405.05932v1 [math.AG] 9 May 2024


\bibitem [Ba]{Ba} K. Banerjee, {\it Algebraic cycles on the Fano variety of a cubic fourfold}, arxiv:1609.05627v1[math.AG], Sept 2016.

\bibitem[BLMNPS]{BLMNPS} Bayer, Lahoz, Macri, Nuer and Stellari, {\it Stability conditions in families}, Publ. Math. IHES 133, 157-325 (2021)

\bibitem [Bois]{Bois}  S.Boissiere,{\it   Automorphismes naturels de l'espace de Douady de points sur une surface}, arXiv:0905.4367v1 [math.AG] 27 May 2009

\bibitem[BCS 1]{BCS 1} S.Boissiere,C.Camere and A.Sarti, {\it Classification of automorphisms on a deformation family of hyper-K\"alher fourfolds with p-elementary lattices}, Kyoto Journal of Mathematics,Vol.56. No.3 (2016) 465-499

\bibitem[BCS 2]{BCS 2} S.Boissiere,C.Camere and A.Sarti, {\it Cubic threefolds and hyper-K\"alher manifolds uniformized by the 10-dimenional complex ball} arXiv:1801-00287 (2017)

 \bibitem[BP]{BP}M.Bolognesi and C.Pedrini,{\it The transcendental motive of a cubic fourfold}, J.Pure Appl. Algebra (2020), https://doi.org/10.1016/ j.jpaa.2020.10633.

\bibitem [Brak] {Brak} E. Brakkee, {\it Two polarized K3 surfaces asociated to the same cubic fourfold}, Math.Proc. Camb.Phil. Soc. 171,  51-64 (2021).

\bibitem[Bull] {Bull}T-H B\"ulles, {\it Motives of moduli spaces on K3 surfaces and of special cubic fourfolds} , Manuscripta Math. (2018) https:// doi.org/10.10007/s00229-018-1086-0.

 \bibitem [Cam]{Cam} C.Camere ,{\it Symplectic involutions of holomorphic symplectic fourfolds}, Bull.Lond. Math. Soc. 44 no. 4(2012),687-702


\bibitem [CT] {CT} J.L. Colliot-Th\'elene,{\it $\CH_0$-trivialite' universelle d' hypersurfaces cubiques presque diagonales}, Algebraic Geometry 4 (5) (2017) 


\bibitem[deC-M]{deC-M} A.de Cataldo and L.Migliorini, {\it The Chow Groups and the Motive of the Hilbert scheme of points on a surface}, Journal of Algebra {\bf 251}, (2002), 824-848.

\bibitem[DM]{DM}I.Dolgachev and D. Markushevich, {\it Lagrangian tens of planes,Enriques surfaces and holomorphic symplectic fourfolds}, arXiv:1906.01445v3 [math.AG] 14 Jan 2020

\bibitem [FV]{FV} G.Farkas and A.Verra, {\it  The moduli space of K3 surfacea of genus 22}, Mathematische Annalen, https://doi.org/10.1007/s00208-020-02036-y (2020)


 \bibitem [Fu]{Fu} L. Fu, {\it Classification of polarized symplectic autormorphisms of Fano varieties of cubic fourfolds}, Glasgow Mathematical Journal, Vol. 58 (2016) 17-37.

  \bibitem [GAL]{GAL} V.Gonzales-Aguillera and A.Liendo,{\it Automorphisms of prime order of smooth cubic n-folds}, Arch.Math.(Basel) 97(2011) no.1 ,25-37.

\bibitem [GS]{GS} A.Garbagnati  and A.Sarti ,{\it Symplectic automorphisms of prime order on K3 surfaces}, J.of Algebra (2007)

\bibitem [GM]{GM} A.Garbagnati and Y.P. Montaniz, {\it Order 3 symplectic automorphism on K3 surfaces},arXiv:2102.01207v2 [math.AG] 22 Sep 2022

\bibitem [GG]{GG} S.Gorchinskiy and V.Guletskii. {\it 
Motives and representability of algebraic cycles on threfolds over a field},
J.Alg. Geometry {\bf 21} (2012) 343-373.



\bibitem[Hass  ]{Hass  }B.Hassett,{\it Cubic fourfolds, K3 surfaces, and rationality questions},Rationality Problems in Algebraic Geometry, 
 R. Pardini and G.P. Pirola, eds., 26-66, CIME Foundation Subseries, Lecture Notes in Mathematics 2172, Springer 2016 

\bibitem[Huy] {Huy} D.Huybrechts,{\it The K3 category of a cubic fourfold-An update} arXiv:2303.03820v1 [math AG] March 2023

 \bibitem[KMP]{KMP} B. Kahn, J. Murre and C. Pedrini, \newblock {\it On the transcendental part of the motive of a surface}, 
pp. 143--202 in "Algebraic cycles and Motives Vol II", 
\newblock London Math. Soc. LNS {\bf 344}, Cambridge Univ 

\bibitem [Kaw]{Kaw} K.Kawatami, {\it On the birational geometry for irreducible symplectic 4-folds related to the Fano schemes of lines},arXiv:0906.0654v1 [math.AG] 3 Jun 2009

 \bibitem [Lat 1 ] {Lat  1} R.Laterveer, {\it  A family of cubics fourfolds with finite-dimensional motive}, Journal of the Math.Soc. of Japan,70 (4) (2017)

\bibitem [Lat 2 ] {Lat  2} R.Laterveer, {\it Algebraic cycles on certain hyperK\"laher fourfolds with an order 3 non-symplectic automorphism II}, arXiv:1802.0703v1 [math.AG] 20Feb 2018

\bibitem [Lat 3 ] {Lat  3} R.Laterveer, {\it Some cubics with finite dimensional motives}, Boll. Soc.Mat.Mex (2018) 24 319-327

\bibitem [Laz] {Laz  } R.Laza, {\it Maximally algebraic potential irrational cubic fourfolds}, Proc.Am.Math. Soc. 149 (8)3209-3220 (2021)

\bibitem [LZ]{LZ} R.Laza and Z.Xheng, {\it Automorphisms and periods of cubic fourfolds},  arXiv:1905.11547v2 [math.AG] 12 Nov 2019

\bibitem [LZP]{LZP}R.Laza, G.Pearlstein and Z.ZHang,{\it On the moduli space of pairs consisting  of a cubic threefold and a hyperplane}, arXiv:1710.08056v2 [math.AG] 25 Oct 2018

\bibitem [Kim]{Kim}S.I.Kimura,{\it Chow groups are finite dimensional in some sense},Math.Ann. (2005). 

 \bibitem[Marq]{Marq} L.Marquand,{\it Cubic fourfolds with an involution},Trans.Am.Math.Soc. 376(2) 1373-1406 (2023)

\bibitem [Mong] {Mong} G.Mongardi, {\it Automorphisms of HyperK\"alherr manifolds} Ph.D Thesis, Universita' Roma 3 (2013)

 \bibitem [Ped ]{Ped  } C. Pedrini, {\it On the finite dimensionality of a K3 surface},  Manuscripta Math. {\bf 138} (2012), 59--72. 

\bibitem [Ou]{Ou} G.Ouchi, {\it Automorphism groups of cubic fourfolds and K3 categories}, arXiv:1909.11033v1 [math.AG] 24 Sep 2019

 \bibitem [RS] {RS}F.Russo and G.Stagliano',{\it Trisecant flops,their associated K3 surfaces and the rationality of some Fano fourfolds}, arXiv:1909.01263v3 [math.AG] Apr.2020

\bibitem [SV ]{SV} M.Shen and C.Vial, {\it The Fourier transform for certain Hyperk\"alher fourfolds}, Memoirs of the AMS 240, no. 1139, (2014), 1-104.


 \bibitem[YY]{YY} S. Yang and X.Yu, {\it On Lattice polarizable cubic fourfolds}, arXiv:2103.09132v1 [math.AG] 16 mar 2021
 


\bibitem [XH]{XH} Xianyu Hu {\it Equivariant  Kuznetsov components of certain cubic fourfolds}. Mater Thesis, Mathematsche Institut Universitaet Bonn, (2022)


 


 \end{thebibliography}
\end{document}